\documentclass[10pt]{amsart}
\pdfoutput=1

\usepackage{amssymb}
\usepackage{amsthm}
\usepackage{amsfonts}
\usepackage{amsmath}
\usepackage{pmboxdraw}
\usepackage{verbatim} 
\usepackage{graphicx}
\usepackage{color, mathtools, textcomp}
\usepackage[colorlinks=true, citecolor=green, filecolor=cyan, linkcolor=red, urlcolor=magenta]{hyperref}
\usepackage{cite}
\usepackage[normalem]{ulem}
\usepackage{enumerate}
\mathtoolsset{showonlyrefs}
\usepackage[bottom]{footmisc}

\def\ms{\medskip}

\def\sm{\setminus}

\def\ti{\tilde}

\def\ve{\varepsilon}
\def\1{\textbf{1}}
\def\eps{\varepsilon}

\def\C{\mathbb{C}}
\def\D{\mathbb{D}}
\def\E{\mathbf{E}}

\def\P{\mathbf{P}}
\def\R{\mathbb{R}}

\def\N{\mathbb{N}}

\theoremstyle{plain}
\newtheorem*{thm*}{Theorem}
\newtheorem{thm}{Theorem}[section]
\newtheorem{lem}[thm]{Lemma}
\newtheorem{cor}[thm]{Corollary}
\newtheorem{prop}[thm]{Proposition}
\newtheorem{definition}{Definition}

\theoremstyle{definition}
\newtheorem*{eg*}{Example}
\newtheorem*{egs*}{Examples}
\newtheorem*{def*}{Definition}

\theoremstyle{remark}
\newtheorem*{rmk*}{Remark}
\newtheorem*{rmks*}{Remarks}

\numberwithin{equation}{section}

\begin{document}
\title[Zeros of random polynomials and its higher derivatives]{Zeros of random polynomials and its higher derivatives}

\author{Sung-Soo Byun } 
\address{\noindent  Department of Mathematical Sciences, Seoul National University,\newline Seoul, 151-747, Republic of Korea}
\email{sungsoobyun@snu.ac.kr}  

\author{Jaehun Lee}
\address{\noindent  Department of Mathematical Sciences, Seoul National University,\newline Seoul, 151-747, Republic of Korea}
\email{hun618@snu.ac.kr}  

\author{Tulasi Ram Reddy}
\address{Division of Sciences, New York University Abu Dhabi, \newline Saadiyat Island, Abu Dhabi, United Arab Emirates}
\email{tulasi@nyu.edu}

\thanks{S.-S. Byun was partially supported by Samsung Science and Technology Foundation (SSTF-BA1401-01). J.-H. Lee was partially supported by the National Research Foundation of Korea (NRF-2016K2A9A2A13003815). Financial support by German Research Foundation (DFG) for S.-S. Byun and J.-H. Lee through the IRTG 2235 is gratefully acknowledged.} 

\begin{abstract}
 In this article we study the limiting empirical measure of zeros of higher derivatives for sequences of random polynomials. We show that these measures agree with the limiting empirical measure of zeros of corresponding random polynomials. Various models of random polynomials are considered by introducing randomness through multiplying a factor with a random zero or removing a zero at random for a given sequence of deterministic polynomials. We also obtain similar results for random polynomials whose zeros are given by i.i.d. random variables. As an application, we show that these phenomenon appear for random polynomials whose zeros are given by the 2D Coulomb gas density. 
\end{abstract}

\maketitle

\section{Introduction}
The study of zeros of polynomials and its derivatives have been of interest for long. A well known result relating these refer to Gauss and Lucas' theorem which states that the critical points of a polynomial  lie in the convex hull formed by the zeros of the polynomial. There have been several extensions and refinements of this result. For instance, Pereira \cite{pereira} and Malamud \cite{malamud} independently extended this result to relating the zeros and critical points with a doubly stochastic matrix. For a deeper discussion on this topic, we refer the reader to classical texts \cite{marden} and \cite{rahman}.

In the theory of random polynomials, one natural way of constructing model is imposing randomness to the coefficients of the polynomials. Then it is of interest to understand the limiting behavior of zero sets when the degree of polynomials goes to infinity. This question was answered by Kabluchko and Zaparozhets in the case that coefficients are taken to be scaled i.i.d. random variables, see \cite{kab_zap}. They showed that the limiting empirical measure of zeros is radially symmetric and depends only on the variance profile of the coefficients. Following the result of Kabluchko and Zaparozhets, one can observe that the empirical measure of zeros of higher derivatives of these polynomials converge to the same measure as that of the zeros of these polynomials.

It is to be noted that when the zeros are all taken to be on the real line, this phenomenon holds for any suitable deterministic sequence of polynomials. This follows from the elementary property that the zeros and critical points of the polynomial interlace. However this phenomenon cease to hold at this generality when the zeros are allowed to be complex numbers. A simple example can be constructed by choosing $P_n(z)=z^n-1$, where the zeros are uniformly distributed on the unit circle, whereas the critical points are accumulated at the origin. For more examples, where this phenomenon does not hold, see the discussion in \cite{Re16b}.

It was first conjectured by Pemantle and Rivin in \cite{PR13} that the limiting empirical measures of zero set and critical point coincide for general random polynomials. Moreover in \cite{PR13}, they studied a model having zeros chosen by i.i.d. random variables and established this phenomenon when the measure $\mu$ from which the zeros are chosen has finite 1-energy. In \cite{sneha}, Subramanian showed a similar result in the case of general probability measure $\mu$ supported on unit circle. Kabluchko extended this phenomenon for general probability measure $\mu$, see \cite{Ka15}.  

In addition, the case that random zeros interact with each other has also been studied. One of the important examples is the characteristic polynomial of some random matrix model. For instance, in \cite{orourke},  O'Rourke showed this phenomenon in the case of characteristic polynomials of circular ensembles. Dennis and Hannay studied the association of critical points of the characteristic polynomial of Ginibre random matrix to its zeros, see \cite{hannay_dennis}. Hanin in \cite{hanin1}, studied the correlation functions of critical points and zeros of spherical polynomial ensembles. It is pertinent to note that whenever the support of the limiting measure of the zeros does not divide the complex plane into disconnected components this phenomenon holds without invoking any randomness, see \cite[Example 4]{eremenko}. As a consequence, this phenomenon holds for the characteristic polynomial of Ginibre random matrix.

It is  natural to study this phenomenon for the higher derivatives too.  We remark that in the case that the limiting empirical measure of zero set is purely atomic, this phenomenon trivially holds. For non-atomic case, Hu and Chang in \cite{chang_hu} studied this problem when the zeros of polynomials are in a bounded strip. 

Our first purpose of this paper is to extend some known results of this phenomena from critical points to zeros of higher derivatives. See Theorem~\ref{iid zero k-th} for the generalization of result in \cite{Ka15} and Theorem~\ref{Ber zero k-th},~\ref{Ber zero array} for those of theorems in \cite{Re16b}.
Our second aim is to extend the class of random polynomials in which this phenomenon occurs by virtue of introducing randomness to a given deterministic polynomials. More precisely, we will impose randomness in such a way as to exclude one zero at random (Theorem~\ref{remove zero}) or to include finite random zeros (Theorem~\ref{t:add}). We remark that Theorem~\ref{remove zero} affirmatively settles the conjecture posed in \cite{Re16b}, when measure is non-atomic. In the last section, we will utilize methods used above to show that this phenomenon also occurs in the case when zeros of the polynomials are given by the 2D Coulomb gas density with general external potential, see Theorem~\ref{Coulomb critical}.

\subsection{Symbols and notation}
\hfill
\ms

Throughout this paper, we will use the following symbols and notation. Let $\D_r$ be the disk of radius $r$, centered at origin. For any polynomial
$P$, we denote by $\mathcal Z(P)$ the multi-set of zeros of $P$ and $\mathcal M(P)$ the uniform probability measure supported on $\mathcal Z(P)$. We write $\delta_a$ the Dirac measure supported at $a$. We recall preliminary definitions introduced in \cite{Re16b}.

\begin{definition}[$\mu$-distributed sequence/triangular array]
	Let $\{a_n\}_{n\ge 1}$ $($resp, $\{a_{n,i}\}_{n\ge1;1\le i\le n})$ be a sequence $($resp., triangular array$\,)$ of complex numbers. If the measure $\frac{1}{n}\sum_{i=1}^n \delta_{a_i}$ $($resp., $\frac{1}{n}\sum_{i=1}^n \delta_{a_{n,i}})$ converges weakly to a probability measure $\mu$, we call such a sequence $($resp, triangular array$\,)$ to be \textit{$\mu$-distributed}. 
\end{definition}

\begin{definition}[$\log$-Ces\'{a}ro bounded sequence/triangular array]
	We say a sequence $($resp, triangular array$\,)$ of complex numbers $\{a_n\}_{n\ge 1}$ $($resp, $\{a_{n,i}\}_{n\ge1;1\le i\le n})$ is \textit{log-Ces\'{a}ro-bounded} if the Ces\'{a}ro means of the positive part of their logarithms are bounded, i.e., the sequence \linebreak $\left \{\frac{1}{n}\sum_{i=1}^n \log_+|a_i| \right \}_{n \ge 1}$, $($resp, $ \left \{\frac{1}{n}\sum_{i=1}^n \log_+|a_{n,i}|\right \}_{n \ge 1})$ is bounded. 
\end{definition}

\subsection{Results}
\hfill 
\medskip

Our first result deals with the random polynomial whose zeros are chosen to be i.i.d. random variables. Note that the empirical measure of these zeros converge to the probability measure from which the random variables are drawn. We show that the same phenomenon appear for the zeros of any derivative of these random polynomials. This result generalizes the result of Kabluchko \cite{Ka15} which was established for zeros of first derivative (critical points).

\begin{thm} \label{iid zero k-th}
	Let  $\{ z_i \}_{i\ge1}$ be i.i.d. random variables distributed according to $\mu$, where $\mu$ is an arbitrary probability measure on $\C$. For each $n\in \mathbb N$, let 
	\begin{equation}
	\displaystyle P_n(z):=(z-z_1)\cdots(z-z_n).
	\end{equation}
	Then for any $k \in \mathbb N$, $\mathcal{M}(P^{(k)}_n) \rightarrow \mu$ in probability. 
\end{thm}

We show that the same result holds when the zeros of polynomials are independently sampled, from two deterministic sequences (triangular arrays), see Theorem \ref{Ber zero k-th}  (Theorem \ref{Ber zero array}) below. We remark that Theorem \ref{Ber zero k-th} can be utilized to verify this phenomenon in the case that the deterministic zeros are perturbed independently at random. For example, choose $a_i=z_i+\sigma_iX_i$ and $b_i=z_i-\sigma_iX_i$, where $\{z_i\}_{i\geq1}$ is a $\mu$-distributed deterministic sequence, $X_i$'s are  i.i.d. symmetric random variables and $\{\sigma_i\}_{i\geq1}$ is a sequence of positive numbers converging to $0$. This is stated as Corollary \ref{perturb zer}. 
These results in the case of zeros of the first derivative were shown in \cite{Re16b}.

\begin{thm} \label{Ber zero k-th}
	Let $\{ a_i \}_{i \ge 1}$ and $\{ b_i \}_{i \ge 1}$ be two $\mu$-distributed, log-Ces\'{a}ro bounded sequences of complex numbers. Suppose that $a_i \neq b_i$ for infinitely many $i$. Let $\{\xi_i\}_{i \ge 1}$ be a sequence of independent random variables such that $\xi_i=a_i$ or $\xi_i=b_i$ with equal probability. For each $n \in \mathbb N$, let 	
	$$
	\displaystyle P_n(z):=(z-\xi_1) \cdots(z-\xi_n).
	$$
	Then $\mathcal{M}(P_n) \rightarrow \mu$ almost surely and $\mathcal{M}(P^{(k)}_n) \rightarrow \mu$ in probability for any $k \in \mathbb N$.
\end{thm}

\begin{cor} \label{perturb zer}
	Let $\{ z_i \}_{\i \ge 1}$ be a $\mu$-distributed log-Ces\'{a}ro bounded sequences of complex numbers. For a non-zero i.i.d. sequence of symmetric random variables $\{ X_i \}$ satisfying $\E[|X_1|] < \infty$, let 
	$$
	\displaystyle P_n(z):= \prod_{i=1}^{n} (z-z_i+\sigma_i X_i), 
	$$
	where $\{ \sigma_i \}_{i\ge1}$ is a decreasing sequence of real number satisfying $\lim_i \sigma_i=0$. Then $\mathcal{M}(P_n) \rightarrow \mu$ almost surely and $\mathcal{M}(P^{(k)}_n) \rightarrow \mu$ in probability for any $k \in \mathbb N$.  
\end{cor}

\begin{thm}\label{Ber zero array}
	Let $\{ a_{i,j} \}_{i \ge 1 ; 1 \le j \le i}$ and $\{ b_{i,j} \}_{i \ge 1 ; 1 \le j \le i}$ be two $\mu$-distributed and log-Ces\'{a}ro bounded triangular arrays of complex numbers satisfying $\sum_{i=1}^n \log_{+} \frac{1}{|a_{n,i}-b_{n,i}|}=o(n^2).$ For each $i \ge 1$, let $\{\xi_{i,j}\}_{j \le i}$ be a sequence of independent random variables such that $\xi_{i,j}=a_{i,j}$ or $\xi_{i,j}=b_{i,j}$ with equal probability. For each $n \in \mathbb N$, let 
	$$
	\displaystyle P_n(z):=(z-\xi_{n,1}) \cdots(z-\xi_{n,n}).
	$$
	Then $\mathcal{M}(P_n) \rightarrow \mu$ almost surely and $\mathcal{M}(P^{(k)}_n) \rightarrow \mu$ in probability for any $k \in \mathbb N$.
\end{thm}

Our next result deals with a question appeared in \cite[Conjecture 2.14]{Re16b} which states that the same phenomenon will happen when a zero is removed uniformly at random from a deterministic sequence of polynomials. We resolve this conjecture positively when the empirical measure of zeros of the polynomials converge to a non-atomic probability measure.

\begin{thm} \label{remove zero}
	Suppose $\{z_i\}_{i \ge 0}$ is a $\mu$-distributed, log-Ces\'{a}ro bounded sequence of complex number, where $\mu$ is a non-atomic probability measure on $\C$. For each $n \in \mathbb N$, let
	$$\displaystyle P_n(z)=\frac{(z-z_0)(z-z_1)\dots (z-z_n)}{z-z_{s_n}}, $$ 
	where $s_n$ is a random number distributed uniformly on the set $\{0,1,\cdots,n\}$. Then $\mathcal{M}(P_n) \rightarrow \mu$ almost surely and $\mathcal{M}(P'_n) \rightarrow \mu$ in probability.
\end{thm}

We now consider sequence of polynomials whose zeros are deterministic except for finite ones. Further we assume that the zero set $\{z_{n,i}\}_{n \ge 1, i \le n}$ is $\mu$-distributed triangular array of complex numbers. For such polynomials, we show that the empirical measure of zeros of higher derivatives (up-to the number of random zeros) converge to the same limiting measure as that of the zeros of these polynomials. We remark that it can be interpreted as a random perturbation of polynomials where the perturbed polynomial is obtained by multiplying with a random factor, which strengthens Theorem 2.1 in \cite{orourke_williams}. 

Before stating our results we introduce the assumptions on the random zeros of our polynomials. Consider a random vector $(X_{n,1}, \dots, X_{n,k})$, where $X_{n,j}$'s are complex-valued random variables distributed according to the joint probability density function $\nu_n(w_1, \dots, w_k)$. Here, we assume that for any $k$, there exist positive constants $C_1,C_2>0$ and $a \in [0,1)$, which does not depend on $n,i$ such that $\nu_n$ satisfies the following conditions. 
\begin{equation}\label{c:add1} 
\int_{\C^k} \sum_{i=1}^k\log_{+}|w_i| \nu_n(w_1,w_2,\dots w_k)dw_1 dw_2 \dots dw_k \le C_1 <\infty;
\end{equation} 
\begin{equation}\label{c:add2}
\frac{\sup_{w_i \in \C} \nu_n(w_1, \dots, w_k)}{ \int_{\C} \nu_n(w_1, \dots, w_k) dw_i } \le C_2 \exp\big( n^a \big) \quad \mbox{for all} \quad w_1,\dots, w_{i-1},w_{i+1}, \dots, w_k \in \C;
\end{equation} 
\begin{equation}\label{c:add3}
\lim_{ r \to \infty}  \limsup_{n \to \infty}\P \left(  \max_{1 \le i \le k}|X_{n,i}| \ge r \right) = 0.
\end{equation}
Notice that \eqref{c:add1} ensures that $X_{n,1}, \dots, X_{n,k}$ have finite $\log_+$-moments. Note also that \eqref{c:add2} has the following probabilistic interpretation: for any given complex numbers $$w_1, \dots, w_{i-1}, w_{i+1}, \dots w_k, $$ the conditional density of $w_i$ is sub-exponentially bounded.

\begin{thm} \label{t: genadd}
	For fixed $k \in \N$ and each $n\in \mathbb N$, suppose that sequence of complex-valued random vector $(X_{n,1}, \dots, X_{n,k})$ with joint probability density $\nu_n(w_1, \dots, w_k)$ satisfies \eqref{c:add1}, \eqref{c:add2} and \eqref{c:add3}. Let
	\begin{equation}
	\displaystyle P_n(z):=(z-z_{n,1})\cdots(z-z_{n,n})(z-X_{n,1}) \cdots (z-X_{n,k}). 
	\end{equation}
	Then  $\mathcal{M}(P_n) \rightarrow \mu$ and $\mathcal{M}(P^{(\ell)}_n) \rightarrow \mu$ in probability for any $1\le \ell \le k$.
\end{thm}

We remark that one of simple examples of probability distributions satisfying above conditions \eqref{c:add1}, \eqref{c:add2} and \eqref{c:add3} is the mutually independent random variables with bounded densities. We state this specific case as the following corollary.  
\begin{cor} \label{t:add}
	Suppose that $\{z_i\}_{i \ge1}$ is a log-Ces\'{a}ro bounded $\mu$-distributed sequence of complex numbers, where $\mu$ is any probability measure on $\C$. Let $k \in \N$ and $X_1, \dots X_k$ be independent complex-valued random variables according to bounded density $\nu_1, \dots, \nu_k$ on $\C$, respectively. For each $n\in \mathbb N$, let 
	$$ \displaystyle P_n(z):=(z-z_1)\cdots(z-z_n)(z-X_1)\cdots(z-X_k). $$
	Then $\mathcal{M}(P_n) \rightarrow \mu$ and $\mathcal{M}(P^{(\ell)}_n) \rightarrow \mu$ in probability for any $1\le \ell \le k$.
\end{cor}

As a consequence of Corollary~\ref{t:add}, Theorem~\ref{iid zero k-th} can be obtained in a special case, when the measure $\mu$ has bounded density and satisfy $\int_{\C}\log_{+}|z|d\mu(z) < \infty$. This is obtained by conditioning on all the zeros except for the first $k$ of them.

We further extend this phenomenon in the case where the zeros of the random polynomials follow 2D Coulomb gas density. For the benefit of the reader we recall some definitions and existing results concerning 2D Coulomb gases. For a fixed positive value $\beta$ and given external field $Q: \C \rightarrow \R$, let $\P_n^{\beta}$ be the point process distributed as
\begin{equation*}
d\mathbf{P}_n^{\beta}(\zeta_1, \cdots, \zeta_n)=  
\frac{1}{ Z_n^{\beta}}  \prod_{j,k:j<k}|\zeta_j-\zeta_k|^{2\beta}e^{-\beta n \sum_i Q(\zeta_i)} d\mbox{vol}_{2n},
\end{equation*}
where $Z_n^{\beta}$ stands for the partition function and $d\mbox{vol}_{2n}$ is the Lebesgue measure in $\R^{2n}$.

As the number of particles goes to infinity, the system $\{\zeta_i \}_{1 \le i \le n}$ tends to be concentrated in a certain compact set $S$ called the droplet. One of the well-known examples is the complex Ginibre ensemble, in which $\beta=1$ and $Q(z)=|z|^2$. In this case the droplet is given as $S= \{ z : |z| \le1 \}$. 
In general, Hedenmalm and Makarov showed that under the mild assumptions on $Q$, the empirical measure of the system $\{\zeta_i \}_{1 \le i \le n}$ converges weakly to the equilibrium measure given by weighted (logarithmic) potential. See \cite{HM13} for more details. Also when $Q$ satisfies some regularity conditions in a neighborhood of $S$, the limiting equilibrium measure $\sigma_Q$ is absolutely continuous with respect to Lebesgue measure $dm$, and takes the following explicit form:
\begin{equation}\label{e:Q}
d\sigma_Q(z)=\frac{1}{4\pi}\,\chi_S \cdot \Delta Q(z) dm(z).
\end{equation}

Before we state our theorem below, we introduce the assumptions on the external potential $Q$. One of the main ingredients in proving Theorem~\ref{Coulomb critical} is a certain type of concentration inequality for 2D Coulomb gas due to Chafa\"{i}, Hardy and Ma\"{i}da, see \cite{CHM17}. Therefore, we also consider the same assumptions on $Q$ as follows. 

\noindent$\bullet$ \textbf{Assumptions (A0).} 
\begin{enumerate}
	\item $Q$ is finite on a set of positive capacity and $C^2$-differentiable;
	\item $\displaystyle \liminf\limits_{|z|\rightarrow \infty} \frac{Q(z)}{|z|^2}>0$;
	\item $\displaystyle \limsup\limits_{|z|\rightarrow \infty} \frac{1}{Q(z)}\sup\limits_{|w-z|<1} \Delta Q(w) <4$.
\end{enumerate}
For basic notions in logarithmic potential theory, we refer the reader to \cite{ST97}. For instance, the potentials $Q(z)=|z|^{2\alpha}$ ($\alpha \ge 1$) satisfies \textbf{(A0)}. Note that (1) implies that $Q$ is admissible and the Boltzmann-Shannon entropy 
$S(\sigma_{Q})= -\int \frac{d\sigma_{Q}}{dx} \log  \frac{d\sigma_{Q}}{dx} dx$
of equilibrium measure $\sigma_Q$ is finite.
We remark that as the authors pointed out, the assumptions \textbf{(A0)} can be weakened as follows, see \cite[Remark 1.10]{CHM17}.

\noindent$\bullet$ \textbf{Assumptions (A1).} 
\begin{enumerate}[(i)] 
	\item $Q$ is finite on a set of positive Lebesgue measure and $\sigma_{Q}$ is of the form \eqref{e:Q};
	\item $\displaystyle \liminf\limits_{|z|\rightarrow \infty} \frac{Q(z)}{|z|^2}>0$;
	\item $Q$ can be decomposed as $Q=\tilde{Q}+h$, where $\tilde{Q}$ is twice differentiable function satisfying condition \textbf{(A0)}-(3), and $h$ is a super-harmonic function.   
\end{enumerate}
Note that by (i) and (iii), $\sigma_{Q}$ has bounded density inside the support, which implies $S(\sigma_{Q})$ is finite. We remark that if $Q$ has a Lipschitz continuous derivative, then $\sigma_{Q}$ is of the form \eqref{e:Q}, see e.g., \cite{CHM17}. In the case of radially symmetric potentials given as $Q(z)=g(|z|)$ for some $g: \R_+ \rightarrow \R$, the following condition implies (i): $Q$ is finite on a set of positive Lebesgue measure and $r g'(r)$ is increasing on $\R_+$ (or $g$ is convex on $\R_+$), see e.g., \cite[IV.6]{ST97}

We emphasize that under our assumption, the ``shape'' of droplets is not restricted to a simply connected domain. For example our theorem cover the case that the external potential is given by ``Mittag-Leffler'' potential
$Q(z)=|z|^{2\alpha}-2\nu\log|z|$, $(\alpha \ge 1,\nu > 0)$. In this case, the droplet is given by annulus where its modulus depends on $\alpha, \nu$. We remark that these cases are not covered by Eremenko's result, see \cite[Example 4]{eremenko}.
\begin{thm} \label{Coulomb critical}
	For any $\beta>0$ and any external potential $Q$ satisfying \rm{\textbf{(A1)}}, let $\{\zeta_i \}_{1 \le i \le n}$ be the corresponding 2D Coulomb gas ensemble, and define 
	$$
	P_n(z):=(z-\zeta_1)\cdots(z-\zeta_n).
	$$
	Then for any $k \in \N$,
	\begin{equation*}
	\mathcal M \big( P^{(k)}_n \big) \rightarrow \frac{1}{4\pi}\,\chi_S \cdot \Delta Q(z) dm(z) \quad \text{in probability.} 
	\end{equation*}
\end{thm}

\section{Outline of method}

To prove our results, we follow the potential theoretic approach introduced by Kabluchko in \cite{Ka15}. For given polynomial $P_n(z)$ with $\mathcal Z(P)=\{ w_1, \cdots, w_n  \}$ and $1 \le k \le n$, let us define 
\begin{equation}\label{L_n^k}
\displaystyle L^k_n(z):=\frac{1}{k!} \frac{P^{(k)}_n(z)}{P_n(z)} = \sum_{1\le i_1 < i_2<\dots<i_k \le n} \frac{1}{z-w_{i_1}}\dots\frac{1}{z-w_{i_k}}.
\end{equation} 
Note that $L_n^k$ is the product of logarithmic derivatives of $P_n,P_n^{(1)},\dots,P_n^{(k-1)}$, with a scaling of $1/k!$. For fixed $k \in \N$ and any $f \in C_c^{\infty}(\C)$ whose support is contained in $\D_r$, we define $
f_n(z):=\frac{1}{n}(\log|L_n^k(z)|)\Delta f(z)
$. Then by applying Green's theorem twice, we have
$$
\displaystyle \frac{1}{2\pi}\int_{\D_r} f(z) \Delta \frac{1}{n} \log |L_n^k(z)|  =\int_{\D_r} f_n(z)dm(z).
$$
Here $\Delta \frac{1}{n} \log |L_n^k(z)|$ is interpreted in the sense of distributions. Notice that,
$$
\frac{1}{2\pi}\int_{\D_r} f(z) \Delta \frac{1}{n} \log |L_n^k(z)|=\frac{1}{n} \sum_{j=1}^{n} f(w_j)-\frac{1}{n} \sum_{j=1}^{n-k} f(\xi^k_j),
$$
where $\{ \xi_j^k : 1 \le j \le n-k \}= \mathcal{Z}(P_n^{(k)})$.
Therefore to verify the concentration of empirical measures $\mathcal{M}(P_n)$ and $\mathcal{M}(P_n^{(k)})$, it is enough to show that $\int_{\D_r} f_n(z)dm(z)$ converges to $0$. To ensure the tightness, we recall a lemma of Tao and Vu.

\begin{lem} {\rm \cite[Lemma 3.1]{TVK10}}. \label{TVK-lem} 
	Let $(X,\mathcal{A},\nu)$ be a finite measure space and $f_n: X \rightarrow \R, n \ge 1$ be  random functions which are defined over a probability space $(\Omega, \mathcal{B},\P)$ and jointly measurable with respect to $\mathcal{A} \otimes \mathcal{B}$. Assume that : 
	\begin{enumerate}	
		\item For $\nu-$a.e. $x\in X$ we have $f_n(x) \rightarrow 0$ in probability, as $n \rightarrow \infty$; 
		\item For some $\delta>0$, the sequence $\int_X |f_n(x)|^{1+\delta}d\nu(x)$ is tight.
	\end{enumerate}
	Then, $\int_X f_n(x)d\nu(x) \rightarrow 0$ in probability, as $n \rightarrow 0$.
\end{lem}

By Lemma~\ref{TVK-lem}, it is enough to show that the following two statements hold.
\begin{equation}\label{A12}
\frac{1}{n}\log\left|L^k_n(z)\right| \rightarrow 0 \quad \mbox{in probability for Lebesgue a.e. }z \in \C;
\end{equation}
\begin{equation}\label{A3}
\text{the sequence} \quad \left \{\int_{\D^r}\frac{1}{n^2}\log^2 \left|L^k_n(z)\right|dm(z) \right \}_{n\ge 1} \quad \text{is tight.}
\end{equation}


\section{Controlling large values of $L_n^k$ and tightness}

As we explained above in Section 2, all we need to show is the upper and lower  estimate of $L_n^k$ and \eqref{A3} for given sequence of random polynomials $P_n$. In this section, we prove the following two lemmas which provide the upper estimate and tightness which can be applied for every cases in our theorems. The following lemma is counterpart of \cite[Lemma 4.2]{Re16b}.
\begin{lem} \label{A1 Re16b}
	Let $\{ a_{i,j} \}_{i \ge1; 1 \le j \le i}$ be any triangular array of numbers. Define  $$\ti{L}_n^{k}(z)=\sum\limits_{1 \le j_1 < \cdots < j_k \le n } \frac{1}{|z-a_{n,j_1}|} \frac{1}{|z-a_{n,j_2}|} \cdots \frac{1}{|z-a_{n,j_k}|}.$$ 
	Then for Lebesgue a.e. $z \in \C$,
	$$
	\displaystyle \limsup_{n \rightarrow \infty} \frac{1}{n} \log \ti{L}_n^{k}(z) \le 0.
	$$		
\end{lem} 
\begin{proof} First, notice that $\log\ti{L}_n^{k}(z) \leq k\log\ti{L}_n^{1}(z)$. 
	Now it is enough to bound $\log\ti{L}_n^{1}(z)$  which follows from the proof of Lemma 4.2 in \cite{Re16b}.
\end{proof}

\begin{lem}
	Let $\{ a_{i,j} \}_{i\ge1; 1\le j \le n} $ be a log-Ces\'{a}ro bounded triangular array of numbers. Define $P_n(z)= \prod_{j=1}^n (z-a_{n,j})$. Then, for any $r>0$ and $k \in \N$, the sequence 
	$$ 
	\left\{ \frac{1}{n^2} \int_{\D_r} \log^2 \left| \frac{P_n^{(k)}(z)}{P_n(z)} \right| dm(z)\right\}_{n\ge k}
	$$ 
	is bounded.
\end{lem}
\begin{proof}
	Notice that
	$$\int_{\D_r} \log^2 \left| \frac{P_n^{(k)}(z)}{P_n(z)} \right| dm(z)=\int_{\D_r} \log_+^2 \left| \frac{P_n^{(k)}(z)}{P_n(z)} \right| dm(z)+\int_{\D_r} \log_-^2 \left| \frac{P_n^{(k)}(z)}{P_n(z)} \right| dm(z).$$
	We now analyse above positive and negative parts of the logarithm separately. For the positive part, observe that by Cauchy-Schwarz inequality, we obtain
	\begin{align*}
	\int_{\D_r} \log_+^2 \left| \frac{P_n^{(k)}(z)}{P_n(z)} \right| dm(z) 
	&= \int_{\D_r} \log_+^2 \left| \frac{P_n^{(k)}(z)}{P_n^{(k-1)}(z)}\dots\frac{P^{(1)}_n(z)}{P_n(z)} \right| dm(z),\\
	&\leq \int_{\D_r} \left( \log_+ \left| \frac{P_n^{(k)}(z)}{P_n^{(k-1)}(z)}\right|+ \cdots + \log_+ \left| \frac{P^{(1)}_n(z)}{P_n(z)} \right| \right)^2 dm(z),
	\\
	&\leq k \int_{\D_r} \log_+^2 \left| \frac{P_n^{(k)}(z)}{P_n^{(k-1)}(z)}\right|+ \cdots + \log_+^2 \left| \frac{P_n^{(1)}(z)}{P_n(z)} \right| dm(z).
	\end{align*}
	Note that for any $0 \le j \le k-1$, each term in the above expression is bounded as 
	\begin{align*}
	&\int_{\D_r} \log_+^2 \left| \frac{P_n^{(j+1)}(z)}{P^{(j)}_n(z)} \right| dm(z) 
	= \int_{\D_r} \log_+^2 \left| \sum_{\omega: P^{(j)}_n(\omega)=0} \frac{1}{z-\omega} \right| dm(z),
	\\
	\leq&  \int_{\D_r} \left(\log(n-j)+\sum_{\omega: P^{(j)}_n(\omega)=0} \log_+ \left| \frac{1}{z-\omega} \right| \right)^2 dm(z),
	\\
	\leq& (n-j+1)\int_{\D_r} \left( \log^2(n-j)+\sum_{\omega: P^{(j)}_n(\omega)=0}  \log_+^2 \left| \frac{1}{z-\omega} \right| \right) dm(z),
	\\
	=&  (n-j+1)\left( \pi r^2 \log^2(n-j)+\int_{\D_r} \sum_{\omega: P^{(j)}_n(\omega)=0} \log_-^2 \left| z-\omega \right| dm(z) \right),
	\end{align*}
	where in the first inequality we have used $\log_+( \sum_{k=1}^n a_k) \le \log n + \sum_{k=1}^n \log_+ a_k$ for any $n \in \N$ and $a_1,\dots,a_n \in \C$, see \cite[Remark 3.2]{Re16b} for instance. By the translation invariance of Lebesgue measure, 
\begin{align}\begin{split}\label{123}
	&\int_{\D_r}\log_-^2|z-\xi|dm(z)=\int_{\D_r+\xi}\log_-^2|z|dm(z) \\ &\le \sup_{K \subset \C} \int_K\log_-^2|z|dm(z)=  \int_{\D_1}\log_-^2|z|dm(z)< \infty.
\end{split}\end{align}
 Therefore, we conclude that $ \left\{ \frac{1}{n^2} \int_{\D_r} \log_+^2 \left| \frac{P_n^{(k)}(z)}{P_n(z)} \right| dm(z)\right\}_{n\ge k}$ is bounded. 
	
	Now we show that the negative part is also bounded. Note that using $\log_-|ab| \le \log_-|a| + \log_-|b|$ for any $a,b \in \C$ we have
	$$
	\frac{1}{n^2}\int_{\D_r}\log^2_-\left| \frac{P_n^{(k)}(z)}{P_n(z)} \right| dm(z) \le \frac{1}{n^2} \int_{\D_r}\log^2_-\left| P_n^{(k)}(z) \right| + \log^2_-\left| \frac{1}{P_n(z)} \right| dm(z). 
	$$
	As in the same way above, we get the following inequality for the first term: 
	$$
	\frac{1}{n^2} \int_{\D_r}\log^2_-\left| P_n^{(k)}(z) \right| dm(z) \le \frac{n-k}{n^2}\int_{\D_r}\sum_{w\in \mathcal Z(P_n^{(k)})} \log_-^2|z-w|dm(z).
	$$
	which is uniformly bounded in $n$ by \eqref{123}. Also, we have
	\begin{align*}
	\frac{1}{n^2}\int_{\D_r} \log^2_-\left| \frac{1}{P_n(z)} \right| dm(z)=& \frac{1}{n^2}\int_{\D_r} \log^2_+\left| P_n(z) \right| dm(z),
	\\
	\leq& \frac{1}{n^2} \int_{\D_r} \left( \sum_{j=1}^n  \log_+\left| z-a_{n,j} \right| \right)^2 dm(z),
	\\
	\leq&  \int_{\D_r} \left( \log 2 +\log_+|z|+\frac{1}{n}\sum_{j=1}^n  \log_+\left|a_{n,j} \right| \right)^2 dm(z).
	\end{align*} 
	Now the lemma follows from the fact that $\{ a_{n,j} \}$ is log-Ces\'{a}ro bounded.
\end{proof}


\section{Controlling small values of $L_n^k$.}

\subsection{Proof of Theorem~\ref{iid zero k-th}, ~\ref{Ber zero k-th} and ~\ref{Ber zero array}.}
\hfill
\medskip

In this subsection, we present the proof of Theorem~\ref{iid zero k-th}, ~\ref{Ber zero k-th} and ~\ref{Ber zero array}. 
Controlling the small values of $L_n^1$ were given in  \cite[Lemma 2.6.]{Ka15} and \cite[Lemma 3.3]{Re16b}. We will use these lemmas and the induction argument to control the small values of $L_n(z)$ in the corresponding  theorems.

\subsubsection{Proof of Theorem~\ref{iid zero k-th}} 
\hfill
\medskip

Recall that under the conditions of Theorem~\ref{iid zero k-th}, zeros of random polynomials $P_n(z)$ are given by i.i.d. random variables $\{ z_i \}_{i\ge1}$ distributed according to $\mu$. First we introduce the following lemma due to Kabluchko.

\begin{lem}{\rm \cite[Lemma 2.6]{Ka15}}
	For Lebesgue a.e. $z \in \C$,
	\begin{equation}\label{e:L_n^1:1}
	\lim_{n\rightarrow \infty }\P\left(|L^1_n(z)| < e^{-n \eps}\right) =0
	\end{equation} 
	for any $\eps>0$, where 
	$L^1_n(z):=\frac{P'_n(z)}{P_n(z)}$.
	Here $P_n(z)$'s are random polynomials in Theorem \ref{iid zero k-th}.
\end{lem}
Recall that $L_n^k$ is given as \eqref{L_n^k} for each $k \le n$. Now, to complete the proof of Theorem~\ref{iid zero k-th}, all we need to show is the following lemma. 
\begin{lem}
	For Lebesgue a.e. $z \in \C$,
	\begin{equation}\label{e:L_n^k:1}
	\lim_{n\rightarrow \infty }\P\left(|L^k_n(z)| < e^{-n \eps}\right) =0
	\end{equation} 
	for every $k \in \N$ and $\eps>0$, where $L^k_n(z):=\frac{1}{k!} \frac{P^{(k)}_n(z)}{P_n(z)}$.
	Here $P_n(z)$'s are random polynomials in Theorem \ref{iid zero k-th}.
\end{lem}

\begin{proof}
	When $\mu$ is purely atomic measure, the proof of \eqref{e:L_n^k:1} is straightforward so without loss of generality we can assume that $\mu$ is not purely atomic. Note that in the case $k=1$, \eqref{e:L_n^k:1} follows from \cite[Lemma 2.6.]{Ka15}. Suppose that \eqref{e:L_n^k:1} holds for some $k \ge 1$. First, we decompose $\mu$ by $\mu_1+\mu_2$ such that $\mu_1$ is purely atomic, and $\mu_2$ is non-atomic. Then $0 \le x:=\mu_1(\C)<1$ since $\mu$ is not purely atomic. Let $E:=supp(\mu_1)$.
	For $j,m<n$, define a random variable $L^{k,j}_n(z)$ by 
	$$ L^{m,j}_n(z):=\sum_{1\le i_1 < \dots < i_m \le n, i_\ell \neq j \forall \ell} \frac{1}{z-z_{i_1}} \dots \frac{1}{z-z_{i_m}}.$$
	Note that $L^{m,j}_n(z)$ is independent of $z_j$. 
	Fix $z \in \C$ which satisfies \eqref{e:L_n^k:1} for every $k \le k_0$ and $\eps>0$. First, for any $\ell \in \N$, by definition of $L_n^{m,j}(z)$ we have
	\begin{equation}\label{1}
	L^{k_0+1}_{n+\ell}(z)=\frac{1}{z-z_j} L^{k_0,j}_{n+\ell}(z) + L^{k_0+1,j}_{n+\ell}(z) \quad \mbox{for all} \quad 1\le j \le \ell. 
	\end{equation}
	For $1\le j \le \ell$, let $\Omega$ be the sample space and denote
	\begin{align*}
		A_j&:= \{\omega \in \Omega : |L^{k_0,j}_{n+\ell}(z)| < e^{-n\eps} \}; \\
	B_j&:= \{\omega \in \Omega : z_j \notin E \}; \\
	C&:= \{\omega \in \Omega : |L^{k_0+1}_{n+\ell}(z)| < e^{-2n\eps}   \},
	\end{align*}
	for fixed $\eps>0$. Note that 	
	$$ \displaystyle \Omega =  \left(\cup_{j=1}^\ell A_j \right) \cup \left(\cup_{j=1}^\ell(A^c_j \cap B_j)\right) \cup \left(\cap_{j=1}^\ell B_j^c\right), $$
	which implies
	\begin{equation*} C \subset \big(\cup_{j=1}^\ell A_j \big) \cup \big(\cup_{j=1}^\ell(A^c_j \cap B_j \cap C)\big) \cup \big(\cap_{j=1}^\ell B_j^c\big). 
	\end{equation*}
	Therefore, we immediately obtain
	\begin{equation}\label{2}
	\displaystyle \P(C)\le \displaystyle \sum_{j=1}^\ell \P(A_j)+\sum_{j=1}^\ell \P(A_j^c \cap B_j \cap C) + \P(\cap_{j=1}^\ell B_j^c).
	\end{equation}
	Note that, since $L^{k_0,j}_{n+\ell}(z)$ and $L^{k_0}_{n+\ell-1}(z)$ are identically distributed, we have 
	\begin{equation}\label{3}
	\P(A_j) = \P\left(|L^{k_0}_{n+\ell-1}(z)|<e^{-n\eps}\right) \quad \mbox{for all} \quad 1 \le j \le \ell.
	\end{equation}
	which converge to $0$ as $n$ goes to $\infty$, for all $1 \le j \le \ell$ by the induction hypothesis. For the second term, using \eqref{1}, we get 
	\begin{align} \begin{split}
	&\P(A_j^c \cap B_j \cap C) 
	\\
	=& \P \left(z_j \notin E, \left|L^{k_0,j}_{n+\ell}(z)\right| \ge e^{-n\eps}, \left|\frac{1}{z-z_j} L^{k_0,j}_{n+\ell}(z) +L^{k_0+1,j}_{n+\ell}(z)\right|<e^{-2n\eps} \right),
	\\ 
	=& \P \left (z_j \notin E, \left|L^{k_0,j}_{n+\ell}(z) \right| \ge e^{-n\eps}, \left|\frac{1}{z-z_j} +\frac{L^{k_0+1,j}_{n+\ell}(z)}{ L^{k_0,j}_{n+\ell}(z)}\right|<\frac{e^{-2n\eps}}{ L^{k_0,j}_{n+\ell}(z)} \right) ,
	\\
	\le& \P \left(z_j \notin E, \left|\frac{1}{z-z_j} +\frac{L^{k_0+1,j}_{n+\ell}(z)}{ L^{k_0,j}_{n+\ell}(z)}\right| < e^{-n\eps} \right). \label{4} 
	\end{split}
	\end{align}
	Let $\nu_z(dw)$ be a distribution of $\frac{1}{z-z_1} 1_{\{z_1 \notin E\}}$. Then since $\nu_z(dw)$ is non-atomic measure, we have 
	\begin{equation}\label{5}
	\lim_{r \downarrow 0} \sup_{z_0 \in \C} \nu_z(B(z_0,r)) =0.
	\end{equation}
	Using \eqref{5} to \eqref{4}, we obtain 
	\begin{align*}
		\P(A_j^c \cap B_j \cap C) &\le \P \left(z_j \notin E, \left|\frac{1}{z-z_j} +\frac{L^{k_0+1,j}_{n+\ell}(z)}{ L^{k_0,j}_{n+\ell}(z)}\right| < e^{-n\eps}\right), 
	\\
	&= \E\left[\nu_z\left(B \big(\frac{L^{k_0+1,j}_{n+\ell}(z)}{ L^{k_0,j}_{n+\ell}(z)},e^{-n\eps} \big) \right) \right], 
	\\
	&\le \sup_{z_0 \in \C} \nu_z(B(z_0,e^{-n\eps})).
	\end{align*}
	Also, since $z_i$ are i.i.d. we have
	\begin{equation}\label{6}
	\P(\cap_{j=1}^\ell B_j^c) = \prod_{j=1}^\ell \P(z_j \in E) = x^\ell.
	\end{equation}
	Combining all \eqref{3}, \eqref{5}, \eqref{6} and \eqref{2}, we obtain
	$$ 
	\limsup_{n\rightarrow \infty}\P(|L^{k_0+1}_n(z)|<e^{-2n\eps})= \limsup_{n \rightarrow \infty} \P(C) \le x^\ell. 
	$$
	for arbitrarily $\ell \in \N$. Now, \eqref{1} for $k=k_0+1$ follows from the fact that $x<1$. This finishes the proof.
\end{proof}

\subsubsection{Proof of Theorem~\ref{Ber zero k-th} and Theorem~\ref{Ber zero array}} 
\hfill
\medskip

In this section we will prove Theorem~\ref{Ber zero k-th} and Theorem~\ref{Ber zero array} at once. Let $\{a_{n,i}\}$ and $\{b_{n,i}\}$ be two $\mu$-distributed log-Ces\'aro bounded triangular arrays of complex numbers. Here we assumed in Theorem~\ref{Ber zero array} that $\sum_{i=1}^n \log_+ \frac{1}{|a_{n,i}-b_{n,i}|} = o(n^2)$. Note that this condition is used only for the case $k=1$. Therefore it suffices to prove Theorem~\ref{Ber zero array} and the proof of Theorem~\ref{Ber zero k-th} follows from similar argument. Recall that zeros of random polynomials $P_n(z)$ are given by the random sequence $\xi_{n,i}$ where $\xi_{n,i}=a_{n,i}$ or $b_{n,i}$ with equal probability. \\

Now following lemma completes the proof of Theorem~\ref{Ber zero array}. 
\begin{lem}\label{l:two2}
	For a.e. $z \in \C$,		
	\begin{equation}\label{e:two2}
	\lim_{n\rightarrow \infty }\P \left( \left|L^k_n(z)\right| < e^{-n \eps} \right) =0
	\end{equation} 
	for all $k \in \mathbb{N}$ and $\eps>0$, where $L^k_n(z):=\frac{1}{k!} \frac{P^{(k)}_n(z)}{P_n(z)}$.
	Here $P_n(z)$'s are random polynomials in Theorem~\ref{Ber zero array}.
\end{lem}

\begin{proof}
	Note that the case $k=1$ is given by \cite[Lemma 4.3 and Lemma 4.4]{Re16b}. In particular, for the proof of \cite[Lemma 4.4]{Re16b} the author has loosened the conditions $\sum_{i=1}^n \log_+ \frac{1}{|a_{n,i}-b_{n,i}|} = o(n^2)$ to the following condition : 
	\begin{equation}\label{e:cond}
	\lim_{\eps \downarrow 0} \liminf_{n \to \infty }\frac{ |\{ i: \log_+ \frac{1}{|a_{n,i}-b_{n,i}|} \le \eps n  \}|}{n} \ge \frac{3}{4}.
	\end{equation}
	Also, \eqref{e:two2} holds for any $z \in \C$ that does not agree with any $a_{n,i}$ or $b_{n,i}$ in the triangular array. 
	
	We will prove Lemma \ref{l:two2} by using induction argument on $k$. 
	Let $$\Omega= \left \{ \xi_n=\{\xi_{n,i}\}_{i\leq n} \, | \, \xi_{n,i} = a_{n,i} \mbox{ or } b_{n,i}, \, n \in \N   \right \}
	$$ be the sample space of all possible finite sequences $\{\xi_{n,i}\}_{i\leq n}$ equipped with uniform probability measure $\P$. By \eqref{e:cond}, we have
	$$ |\{ i: a_{n,i} \neq b_{n,i}, i \le n \}| \to \infty \quad \mbox{as} \quad n \to \infty. $$
Thus, for any $\ell \in \N$, we may assume that $a_{n,i} \neq b_{n,i}$ for $1 \le i \le \ell$ for large $n$, since finite permutations on sequences does not affect \eqref{e:two2}.
	
	Fix $z \in \C$, $\ell,k \in \N$, and let $\mathcal{N}_n=\mathcal{N}_n(z)$ be a set of finite sequence $\xi_n=\{\xi_{n,i}\}_{i \le n}$ such that
	$$\mathcal{N}_n:= \left \{ \xi_n \in \Omega \, |\, \  \,  \left|L^{k}_n(z) \right| < e^{-n\eps}  \right \}. 
	$$
	Note that for any $n \in \N$ and $z \in \C$, the value of $L_n^{k}(z)$ depends only on $\xi_n$. Set
	\begin{equation}\label{e:7}
	\mathcal{N}_n^\ell := \left\{  \xi_n \in \mathcal{N}_{n} \, |\, \exists \eta_n \in \mathcal{N}_n \mbox{ such that }\eta_n \neq \xi_n, \eta_{n,i} = \xi_{n,i} \mbox{ for all } l < i \le \ell    \right\}.
	\end{equation}
	Note that if there exist two sequences $\xi_n, \ti{\xi}_n$ in $\mathcal{N}_n \backslash \mathcal{N}^\ell_n$ with $\xi_{n,i}=\ti{\xi}_{n,i}$ for $1 \le i \le \ell$, then $\xi_n=\ti{\xi}_n$ by the construction. Therefore for each given sequence of tails $\{\xi_{n,i} \}_{\ell <i \le n}$, there is at most one sample contained in $\mathcal{N}_n \backslash \mathcal{N}^\ell_n$, which implies 
	\begin{equation}\label{e:11}
	\P(\mathcal{N}_n \backslash \mathcal{N}^\ell_n) \le 2^{-\ell}.
	\end{equation} 
	For any $n \in \N$ and $\xi_n \in \mathcal{N}^\ell_n$, using the definition we can find a sequence $\eta_n \in \mathcal{N}^\ell_n$ such that $\eta_n \neq \xi_n$ and $\eta_{n,i} = \xi_{n,i}$ for all $l<i \le \ell$. Note that since $\xi_n, \, \eta_n \in \mathcal{N}_n$, 
	\begin{equation}\label{e:3}
	\left|L^k_n(z) (\xi_n ) \right|<e^{-n\eps}, \quad \quad \left|L^k_n(z)(\eta_n ) \right|<e^{-n\eps}.
	\end{equation}
	Let $F=\{ i_j \}_{1 \le j \le m}$ be the set of index such that $\xi_{n,i_j} \neq \eta_{n,i_j}$. Without loss of generality, we may assume that $1 \le i_1<i_2<...<i_m \le \ell$. 
	For $1\le p \le m$, let us denote  
	$$\alpha_p:= \sum_{1 \le j_1<\dots<j_p \le m} \frac{1}{z-\xi_{n,i_{j_1}}}\dots\frac{1}{z-\xi_{n,i_{j_p}}}, \quad \beta_p:=\sum_{1 \le j_1<\dots <j_p \le m} \frac{1}{z-\eta_{n,i_{j_1}}}\dots\frac{1}{z-\eta_{n,i_{j_p}}}, 
	$$ 
	and $\alpha_0=\beta_0=1$. 
	Note that, by construction, 
	$$
	\displaystyle \sum_{p=0}^m \alpha_p w^{m-p}=\prod_{j=1}^m \left(w+ \frac{1}{z-\xi_{n,i_j}}\right), \quad  \sum_{p=0}^m \beta_p w^{m-p}=\prod_{j=1}^m \left(w+ \frac{1}{z-\eta_{n,i_j}}\right), 
	$$
	which implies that two polynomials  $\sum_{p=0}^m \alpha_p w^{m-p}$ and  $\sum_{p=0}^m \beta_p w^{m-p}$ cannot be the same as polynomials in `$w$'. Therefore, $\alpha_p \neq \beta_p$ for at least one of $1 \le p \le m$. 
	
	For each $k_1 \in \N$ and $E \subset \{1,2,\dots,n\}$, let us define $L^{k_1,E}_n(z)$ by 
	$$L^{k_1,E}_n(z)=L^{k_1,E}_n(z)(\xi_n) :=\sum_{1\le i_1 < \dots < i_{k_1} \le n, i_j \notin E\, \forall j} \frac{1}{z-\xi_{n,i_1}} \dots \frac{1}{z-\xi_{n,i_k}}.
	$$
	Note that $L^{k_1,E}_n(z)$ is independent of $\{\xi_{n,i}\}_{i \in E}$ by the independence of $\xi_{n,i}$. Also by definition of $\alpha_p$, $\beta_p$ and $F$, we have the following decomposition of $L^{k}_n(z)$:
	\begin{equation}
	\label{e:4}
	L^{k}_n(z)(\xi_n )=\sum_{p=0}^m \alpha_p L^{k-p,F}_{n}(z)(\xi_n ), \quad L^{k}_n(z)(\eta_n )=\sum_{p=0}^m \beta_p L^{k-p,F}_{n}(z)(\eta_n ).
	\end{equation}
	Notice that since $\xi_{n,i} = \eta_{n,i}$ for all $i \notin F$, $L^{k-p,F}_{n}(z)(\xi_n)= L^{k-p,F}_{n}(z)(\eta_n)$ for all $0 \le p \le m$. Therefore by \eqref{e:3} and \eqref{e:4}, we get
	\begin{align}
	\begin{split}\label{e:5}
	2e^{-n\eps} &\quad \ge\quad  \left |L_n^{k}(z)(\xi_n)-L_n^{k}(z)(\eta_n) \right|, \\
	&\quad = \quad \left |\sum_{p=0}^m \alpha_p L^{k-p,F}_{n}(z)(\xi_n) - \sum_{p=0}^m \beta_p L^{k-p,F}_{n}(z)(\eta_n)\right|,  \\
	&\quad = \quad  \left |\sum_{p=1}^m (\alpha_p- \beta_p) L^{k-p,F}_{n}(z)(\xi_n) \right|,
	\end{split}\end{align}
	for sufficiently large $n$.
	
	Let $m_0:= \max\{1 \le p \le m: \alpha_p \neq \beta_p \}$, $m_1:= \min\{1 \le p \le m: \alpha_p \neq \beta_p \}$ and $\kappa_1,\dots,\kappa_q$ be zeros of polynomial 
	$$ f(w):=\sum_{p=m_1}^{m_0} (\alpha_p - \beta_p)w^{m_0 - p}. $$ 
	Note that $q \le m-1$ since $q=\deg f = m_0 - m_1 \le m-1$. Also the fact that $f(0)= \alpha_{m_0} - \beta_{m_0} \neq 0$ implies $\kappa_i \neq 0$ for every $1 \le i \le q$. Let $\gamma_i= z+ \frac{1}{\kappa_i}$ for each $1 \le i \le q$. Then we obtain
	\begin{equation}
	f(w)= \prod_{i=1}^q \left(w+\frac{1}{z-\gamma_i}\right), \quad \alpha_p-\beta_p= L_q^p(\{\gamma_i\}_{i \le q}). 
	\end{equation}
	Now let us define a random sequence $\{\nu_{n,i}\}_{i \le n-m+q}$ by 
	\begin{equation}\label{d:nu}
	\nu_{n,i} =\begin{cases} \gamma_i &\mbox{for } \quad 1 \le i \le q \\ 
	\xi'_{i-q} & \mbox{for } \quad q < i \le n-m+q,
	\end{cases}
	\end{equation}
	where the random sequence $\{\xi'_i\}_{1 \le i \le n-m}$ is constructed as follows: 
	\begin{itemize}
		\item For $1 \le i \le \ell-m$, by shortening the sequence $\{\xi_i\}_{i \in \{1,\dots,\ell\}} $ to $\{\xi_i\}_{i \in \{1,\dots,\ell\} \backslash F}$; 
		\item For $i > \ell-m$, pick $a_{n,i+m}$ or $b_{n,i+m}$ with equal probability.  
	\end{itemize} 
	Combining all above, we obtain
	\begin{align}
	\begin{split}\label{e:6}
	\left|\sum_{p=1}^m (\alpha_p-\beta_p) L^{k_0-p,F}_{n}(z) \left(\xi_n \right) \right|
	&\quad =\quad  \left|\sum_{p=m_1}^{m_0} (\alpha_p-\beta_p) L^{k_0-p,F}_{n}(z) \left(\{\xi_{n,i}\}_{i \le n} \right) \right| ,
	\\
	&\quad =\quad \left|\sum_{p=0}^q L_q^p (z)\left(\{\gamma_i\}_{i \le q}\right) L_n^{k-m_1-p,F}(z)\left(\{\xi_{n,i}\}_{i \le n} \right)\right| ,
 \\
	&\quad =\quad \left|\sum_{p=0}^q L_q^p (z)\left(\{\gamma_i\}_{i \le q}\right) L_{n-m}^{k-m_1-p}(z)(\{\xi'_i\}_{i \le n-m}) \right|,
 \\
	&\quad =\quad \left|L^{k-m_1} _{n-m+q}(z)\left(\{\nu_{n,i}\}\right) \right|. 
	\end{split}
	\end{align}
	Note that since $\ell$ is fixed number, there are only finite cases of $\{\nu_{n,i}\}_{i \le n-m+q}$. More precisely, since the deterministic part of $\{\nu_{n,i}\}_{i \le n-m+q}$ is determined by the values of $\xi_{n,1},\dots,\xi_{n,\ell}$ and $\eta_{n,1},\dots,\eta_{n,\ell}$, the number of cases is less than $2^{2\ell}$ for each $n$. Now we inductively define $\mathcal{Z}_{k,n,\ell}$ and $C_{k,n,\ell}$ as follows.
	\begin{enumerate}[(i)]
		\item First we set 
		$$C_{1,n,\ell}:= \{ \{(a_{n,i},b_{n,i}) \}_{i \le n}  \}\quad \mbox{and} \quad \mathcal{Z}_{1,n,\ell}:= \{a_{n,i},b_{n,i} : \{(a_{n,i},b_{n,i}) \}_{i \le n} \in C_{1,n,\ell} \}.$$ 
		\item We denote by $C_{2,n,\ell}$ the collection of $ \{\nu_{n,i}\}$, where $\nu$ is a sequence of 2-vector defined by \eqref{d:nu} with the sequence of 2-vector $\xi \in C_{1,n,\ell} $ and set 
		$$\mathcal{Z}_{2,n,\ell}:= \mathcal{Z}_{1,n,\ell} \cup \{ \nu^1_{n,i}, \nu^2_{n,i} : \{\nu_{n,i}\} = \{ (\nu^1_{n,i}, \nu^2_{n,i})\} \in C_{1,n,\ell}    \}. $$
		\item Finally, for any $k \ge 1$, we define $C_{k+1,n,\ell}$ as the collection of $ \{\nu_{n,i}\}$, where $\nu$ is a sequence of 2-vector defined by \eqref{d:nu} with the sequence of 2-vector $\xi \in C_{k,n,\ell} $ and set 
		$$\mathcal{Z}_{k+1,n,\ell}:=\mathcal{Z}_{k,n,\ell} \cup \{ \nu^1_{n,i}, \nu^2_{n,i} : \{\nu_{n,i}\} = \{ (\nu^1_{n,i}, \nu^2_{n,i})\} \in C_{k+1,n,\ell}    \}. $$
	\end{enumerate} 
	Note that for each $k,n,\ell \in \N$, $|C_{k,n,\ell}| \le 2^{2k\ell}$ since there is at most $2^{2 \ell}$ choice of $\{\nu_{n,i}  \}$ for each sequence of 2-vector. Thus, $Z_{k,n,\ell}$ is also finite set. Let
	$$ \mathcal{Z}_{k} := \bigcup_{n=1}^\infty \bigcup_{\ell=1}^\infty \mathcal{Z}_{k,n,\ell},$$
	which will be exceptional set of \eqref{e:two2} for $k$. Note that $\mathcal{Z}_{k} \subset \C$ is countable set since each $\mathcal{Z}_{k,n,\ell}$ is finite. Also,  $\mathcal{Z}_{k}$ is increasing in $k$ since $\mathcal{Z}_{k,n,\ell}$ is increasing in $k$ for each $n$,  $\ell$. Now we are ready to state our induction hypothesis. We claim that for any $\{a_{n,i}\}$ and $\{b_{n,i}\}$ satisfying \eqref{e:cond}, \eqref{e:two2} holds for any $k \in \N$ and $z \in \mathcal{Z}_k^c$. Let us emphasize that we already proved the case $k=1$ at the beginning of this proof. Assume that our claim is valid for $k=1,2,\dots,k_0 - 1$. Fix $w \in \mathcal{Z}^c_{k_0}$ and $\ell \in \N$, and define $\mathcal{N}_n$ and $\mathcal{N}_n^\ell$ as above. Note that it suffices to show that
	$$ \lim_{n \to \infty} \P(\mathcal{N}_n  ) = \lim_{n \to \infty} \P(\mathcal{N}_n(w)  )=0  $$
	By \eqref{e:11}, \eqref{e:5} and \eqref{e:6} we obtain
	\begin{align*}
	&\limsup_{n \rightarrow \infty} \P(\mathcal{N}_n) \le \limsup_{n \rightarrow \infty} \P(\mathcal{N}_n \backslash \mathcal{N}_n^\ell ) + \limsup_{n \rightarrow \infty} \P( \mathcal{N}_n^\ell ) \le 2^{-\ell} + \limsup_{n \rightarrow \infty} \P( \mathcal{N}_n^\ell ) ,\\
	&\le 2^{-\ell} + \limsup_{n \to \infty} \P\left(|L_n^{k_0}(w)(\xi_n)|<e^{-n \eps} \right),\\
	&\le 2^{-\ell} +	\limsup_{n \to \infty}\P\left(\left|L^{k_0-m_1}_{n-m+q}(w)\left(\nu_n \right)\right| < 2e^{-n\eps} \mbox{ for some } \nu_n=\{\nu_{n,i}\} \in C_{2,n,\ell} \right), \\
	&\le 2^{-\ell} + \sum_{\nu_n \in C_{2,n,\ell}} 	\limsup_{n \to \infty}\P\left(\left|L^{k_0-m_1}_{n-m+q}(w)\left(\nu_n \right)\right| < 2e^{-n\eps} \right).
	\end{align*}
	By the definition of $C_{k_0,n,\ell}$ and induction hypothesis, we first  have $|C_{2,n,l}| \le 2^{2 \ell}$. 
	
	For the sequence of 2-vector $\xi= \{\xi_n \}_{n \ge 1} = \{  (\xi^1_{n,i}, \xi^2_{n,i}    )   \}_{i \le n, n \ge 1}$, let us denote by $\mathcal{Z}_{k,n,\ell}(\xi_n)$ the set defined same as $\mathcal{Z}_{k,n,\ell}$, except that $C_{1,n,\ell}= \{ \xi_n  \}$. Then we observe that for any $m \ge 1$ and $\nu_n \in C_{2,n,\ell}$,
 $$\mathcal{Z}_{k_0-m,n,\ell}(\nu_n) \subset \mathcal{Z}_{k_0-m+1,n,\ell} \subset \mathcal{Z}_{k_0,n,\ell}.$$
Thus, we conclude that since $w \notin \mathcal{Z}_{k_0}$ and $m_1 \ge 1$,
	$$ \lim_{n \to \infty}\P\left(\left|L^{k_0-m_1}_{n-m+q}(w)\left(\nu_n \right)\right| < 2e^{-n\eps}  \right)=0 \quad \mbox{for any} \quad \nu_n \in C_{2,n,\ell}.$$
  Therefore, 
  $$ \limsup_{n \to \infty} \P(\mathcal{N}_n^\ell)=0.$$
	Now the claim follows from the fact that $\ell$ is arbitrary, which completes the proof.
\end{proof}

\subsection{Proof of Theorem~\ref{remove zero} and ~\ref{t: genadd}} \label{1.5,6}
\hfill
\medskip

In this subsection, we prove Theorem~\ref{remove zero} and ~\ref{t: genadd}. Recall that in these cases, considering random polynomials are constructed as giving randomness to deterministic ones.

\subsubsection{Proof of Theorem~\ref{remove zero}}
\hfill
\medskip

Let $\{a_{n,i}\}_{n\ge 1, 0\le i \le n}$ be a triangular array defined by
$a_{n,i}:=1-\delta_{s_n}(i)$, where $s_n$ is a random number distributed uniformly on the set $\{0,1,\cdots,n\}$. Assume that $\mu$ is non-atomic probability measure. Let $\{z_n\}_{n \ge 0}$ be a $\mu$-distributed and log-Cesaro bounded sequence.
\begin{lem}
	For Lebesgue a.e. $z \in \C$,
	\begin{equation}
	\lim_{n\rightarrow \infty }\P(|L^1_n(z)| < e^{-n \eps}) =0
	\end{equation} 
	for every $\eps>0$, where $L^1_n$ is given as 
	$$ L^1_n(z):= \frac{a_{n,0}}{z-z_0} + \frac{a_{n,1}}{z-z_1} + \dots + \frac{a_{n,n}}{z-z_n}.
	$$ 
\end{lem}
\begin{proof} Since $\mu$ is non-atomic, it suffices to show that if 
	$\limsup_{n\rightarrow \infty} \P(|L_n(z)|<e^{-n\eps})>0$ for some $\eps>0, z\in \C,$ there exist $x\in \C$ with $\mu(x)>0$. Suppose that 
	$$\displaystyle \limsup_{n\rightarrow \infty} \P(|L_n(z)|<e^{-n\eps})=3\delta>0
	$$ 
	for some fixed $\eps>0, z\in \C$.
	Then, there exists a subsequence $\{n_k\}_{k\ge1}$ satisfying 	
	\begin{equation}\label{A2:1}
	\displaystyle \P \left(|L_{n_k}(z)|<e^{-n_k\eps}\right)>2\delta
	\end{equation}
	Since the value $L_n(z)$ is just determined by $s_n \in \{0,1,\dots,n\}$, \eqref{A2:1} implies that 
	\begin{equation} 
	\label{A2:2}
	\left| 0\le i\le n_k : \big|\sum_{j=0}^{n_k} \frac{1}{z-z_j} - \frac{1}{z-z_i} \big|<e^{-n_k\eps} \right| > 2\delta (n_k+1).
	\end{equation}
	Let 
	$$N_k:=\left|0\le i\le n_k : \left|\sum \limits_{j=0}^{n_k} \frac{1}{z-z_j} - \frac{1}{z-z_i} \right|<e^{-n_k\eps} \right|,
	$$
	and $w_{k,1},w_{k_2},\dots,w_{k,N_k}$ are those values of $z_i$'s, i.e.,
	$$
	\displaystyle \left|\sum \limits_{j=0}^{n_k} \frac{1}{z-z_j} - \frac{1}{z-w_{k,j}} \right|<e^{-n_k\eps}, \quad 1 \le j \le N_k.
	$$ 
	Note that $w_{k,i}$ may have same values. Without loss of generality, we may assume $|w_{k,1}| \le |w_{k,2}| \le \dots \le |w_{k,N_k}|.$ Then by \eqref{A2:2}, for all $1\le i_1,i_2 \le N_k$,
	\begin{equation}\label{A2:3}
	\displaystyle  \left|\frac{1}{z-w_{k,i_1}}-\frac{1}{z-w_{k,i_2}}\right| < 2e^{-n_k \eps}.
	\end{equation}
	
	We will use the following inequality: 
	\begin{equation}\label{log ineq.} 
	\displaystyle \frac{1}{n_k}\sum_{i=1}^{n_k} \log_+ |z_i| \ge \frac{1}{n_k}\sum_{i=1}^{N_k} \log_+ |w_{k,i}| \ge \frac{N_k}{n_k} \log_+ |w_{n,1}| \ge 2\delta \log_+|w_{n,1}|.
	\end{equation}
	Notice that since $\{z_k\}$ is log Ces\'{a}ro bounded, $\log_+|w_{n,1}|$ is also bounded.
	
	Let $B_k:={\bar B(w_{k,1},r_k)}$ be closed ball centered at $w_{k,1}$ with radius 
	$$\displaystyle r_k:=\frac{2\exp(-{n_k}   \eps){R_k}^2}{1-2R_k \exp(-{n_k}\eps)}, \quad  R_k:=|z-w_{k,1}|.
	$$
	By \eqref{log ineq.}, $R_k$ is bounded, which implies $1>r_k>0$ for large $k$. Therefore, we may assume that $1>r_k>0$ for all $k\ge 1$ without loss of generality.
	Note that if $w \notin B_k$,  
	\begin{equation}\label{A2:4}
	\begin{array}{lcl}
	\displaystyle \left|\frac{1}{z-w_{k,1}}-\frac{1}{z-w}\right|&=& \displaystyle \left|\frac{w-w_{k,1}}{(z-w_{k,1})(z-w_k)}\right| ,
	\\
	\\
	&\ge& \displaystyle \left|\frac{r_k}{(z-w_{k,1})(z-w)}\right| \ge \frac{r_k}{R_k(R_k+r_k)} \ge 2e^{-n_k \eps}.
	\end{array}
	\end{equation}
	\\
	Then by \eqref{A2:3} and \eqref{A2:4}, $w_{k,i} \in B_k$ for all $1 \le i \le N_k$.
	
	Let $\displaystyle \mu_k:= \frac{1}{n_k+1} \sum_{i=0}^{n_k} \delta_{z_i}.$ Then, $\mu_k(B_k)=\frac{N_k}{n_k+1} \ge 2\delta$ and $\mu_k \rightarrow \mu$ weakly. Note that if $\mu_k \rightarrow \mu$ weakly, $\displaystyle \limsup_{k\rightarrow \infty}\mu_k(C) \le \mu(C)$ for all closed set C and $\displaystyle \liminf_{k\rightarrow \infty}\mu_k(U) \ge \mu(U)$ for all open set U.  
	
	We can find $\tilde{R}>0$ satisfying $\mu(B(0,\tilde{R}-2)) > 1-\delta/3.$ Since $\displaystyle \liminf_{k\rightarrow \infty}\mu_k(B(0,\tilde{R}-2)) \ge \mu(B(0,\tilde{R}-2))$, $\mu_k(B(0,\tilde{R}-2)) > 1-\delta/2.$ for large k. So $B_k \cap B(0,\tilde{R}-2) \neq \emptyset$. Since $r_k<1$, we can say that $B_k \subset \bar{B}(0,\tilde{R}):=K$ for large k. 
	
	Suppose $\mu$ has no point measure. Then, for all $x\in K$, there exist $\eps=\eps(x)>0$ such that $\mu(\bar{B}(x,\eps(x)))<\delta.$ $\{B(x,\frac{\eps(x)}{2})\}_{x\in K}$ makes open cover of compact set K and there exist finite open cover $\{B(x_i,\frac{\eps(x_i)}{2})\}_{i\le N}.$
	
	Let $\eps=\min\{ \eps(x_i)/2 : i \le N\}$. $r_k<\eps$ for large k since $r_k \downarrow 0.$ For each k with $r_k<\eps$ , there exists $i\le N$ satisfying $B_k \cap B(x_i, \eps(x_i)/2) \neq \emptyset$. So, $B_k \subset \bar{B}(x_i,\eps(x_i))$ for some $i$. Therefore, there exists $i\le N$ such that 	
	\begin{equation}\label{A2:5}
	\displaystyle B_k \subset \bar{B}(x_i,\eps(x_i))=:C\mbox{ for infinitely many }k. 
	\end{equation}
	Therefore, $2\delta \le \limsup_k \mu_k(C) \le \mu(C) < \delta$ is contradiction, where we used \eqref{A2:5} for the first inequality, and the fact $C$ is closed and $\mu_k \rightarrow \mu$ for the second inequality. Therefore, $\mu$ has point mass and the conclusion of the lemma follows.
	
\end{proof}

\subsubsection{Proof of  Theorem~\ref{t: genadd}}
\hfill
\medskip

Now we prove Theorem~\ref{t: genadd}. Recall that sequence of complex-valued random vector $(X_{n,1}, \dots, X_{n,k})$ with joint probability density $\nu_n(w_1, \dots, w_k)$ satisfies \eqref{c:add1}, \eqref{c:add2} and \eqref{c:add3}. The following lemma shows that \eqref{A12} holds for $L_n^1(z)$.
\begin{lem}\label{l:add1}
	For Lebesgue a.e. $z \in \C$,
	\begin{equation}\label{e:L_n^k:2}
	\lim_{n\rightarrow \infty }\P(|L^1_n(z)| < e^{-n \eps}) =0
	\end{equation} 
	for every $\eps>0$, where 
	$L^1_n(z):=\frac{P'_n(z)}{P_n(z)}$. Here $P_n(z)$'s are random polynomials given as 
	\begin{equation}\label{pn}
	\displaystyle P_n(z):=(z-z_{n,1})\cdots(z-z_{n,n})(z-X_{n,1}) \cdots (z-X_{n,k}). 
	\end{equation}.
\end{lem}
\begin{proof}
	First we claim that it suffices to show the case $k=1$. Assume that $\eqref{e:L_n^k:2}$ holds for $k=1$. Now for any fixed $X_{n,1}, \dots, X_{n,k-1}$ with general $k \ge 2$, we can define $w_{n+k-1,i}= z_{n,i}$ for $1 \le i \le n$, $w_{n+k-1,n+j} = X_{n,j}$ for $1 \le j \le k-1$ and $Y_{n,1}=X_{n+k,1}$ so that $\{w_{n,i}\}_{n \in \N, 1 \le i \le n}$ is $\mu$-distributed and $Y_{n,1}$ satisfies \eqref{c:add2} and \eqref{c:add3}. Thus we have $\eqref{e:L_n^k:2}$ for any fixed $X_{n,1}, \dots, X_{n,k-1}$, which implies $\eqref{e:L_n^k:2}$. So we can assume $k=1$ without loss of generality. 
	
	Fix $\delta >0$ and $z \in \C$ satisfying $z \neq z_{n,i}$ for any $n \in \N$ and $1 \le i \le n$. By \eqref{c:add3}, we have constants $r_\delta>0$ and $N_\delta \in \N$ such that
	$$ \P(|X_{n,1}| \ge r_\delta - |z|) \le \frac{\delta}{2} \quad \mbox{for} \quad n \ge N_\delta. 
	$$
	Let $\theta^{(n)}_z(w)dw$ be the conditional density of $\frac{1}{z-X_{n,1}} \1_{\{ |X_{n,1}| \le r_\delta - |z|  \}}$ given the others. Note that by \eqref{c:add2} we have
	$$   \theta^{(n)}_z(w) \le \frac{|z-X_{n,1}|^2 \sup_{w_i \in \C} \nu_n(w)}{\int_{\C} \nu_n(w) dw } \le  C_2  e^{cn^a}|z-X_{n,1}|^2 \le C_2 r_{\delta}^2 e^{cn^a}.$$ 
	By definition of $L_n^1$, we may write $$L_n^1(z)=\frac{1}{z-X_{n,1}} +  \sum_{i=1}^n \frac{1}{z-z_i} .$$
	Let us denote
	$$ Y_n := \frac{1}{z-X_{n,1}} \quad \mbox{and} \quad x_n:= \sum_{i=1}^n \frac{1}{z-z_i},
	$$
	i.e., $L_n^1(z) = Y_n + x_n$. Then for large $n$ satisfying $n \ge N_\delta$ and $\pi C_2 r_\delta^2  e^{-2n\eps+cn^a} \le \delta/2$, we have
	\begin{align*}
	\P(L_n^1(z) \le e^{-n\eps}) &\le \P(|X_{n,1}| \ge r_\delta - |z|) + \P( |Y_n+x_n| \le e^{-n\eps} ; |X_{n,1}| \le r_\delta - |z|) ,\\ &\le \frac{\delta}{2} + \sup_{x \in \C} \P( |Y_n| \le B(-x,e^{-n\eps});|X_{n,1}| \le r_\delta - |z|), \\
	&\le \frac{\delta}{2} + C_3 e^{-2n\eps+cn^a} \le \delta.
	\end{align*}
	for some constant $C_3$. Since $\delta>0$ is arbitrary, we obtain \eqref{e:L_n^k:2}. 
\end{proof} 

Now we are ready to use induction on $k$ to obatin \eqref{A12}. The following lemma completes the proof of Theorem~\ref{t: genadd}. 

\begin{lem}\label{l:add2}
	Let $k \in \N$ be the number of $X_{n,j}$'s. For Lebesgue a.e. $z \in \C$, we have
	\begin{equation}\label{e:L_n^k:3}
	\lim_{n\rightarrow \infty }\P(|L^\ell_n(z)| < e^{-n \eps}) =0
	\end{equation} 
	for every $\eps>0$ and $\ell \le k$, where $L^\ell_n(z):=\frac{1}{\ell!} \frac{P^{(\ell)}_n(z)}{P_n(z)}$.
	Here $P_n(z)$'s are random polynomials defined by \eqref{pn}.
\end{lem}
\begin{proof} Repeating the argument of Lemma \ref{l:add1}, we can obtain that it suffices to prove Lemma \ref{l:add2} for the case $\ell=k$.
	Define 
	$$ 
	Q_n(z):= (z-z_{n,1})\cdots(z-z_{n,n})(z-X_{n,1})\cdots(z-X_{n,k-1})= \frac{P_n(z)}{z-X_{n,k}}.
	$$  Then we observe that
	$$ 
	L_n^k(z) = \frac{1}{z-X_{n,k}} M^{k-1}_n(z) + M^k_n(z), 
	$$
	where $ M^l_n(z):= \frac{1}{l!}\frac{Q_n^{(l)}(z)}{Q_n(z)}$ for $l \in \N$. Thus, for any $\eps,\delta>0$ and any $z \in \C$ satisfying $z \neq z_{n,i}$ for all $n,i$, we have	
	\begin{align*}
	\P\left(|L^k_n(z)| < e^{-2n \eps}\right) &=\P \left( \left| \frac{M^{k-1}_n(z)}{z-X_{n,k}}+M^k_n(z) \right| < e^{-2n \eps}\right) ,\\ &\le  \P \left( \left| \frac{1}{z-X_{n,k}}+\frac{M^k_n(z)}{M^{k-1}_n(z)} \right| < e^{-n \eps} ; |X_{n,k}| \le r_\delta - |z| \right) ,
	\\
	&+\P(|M_n^{k-1}(z)| < e^{-n\eps}) + \P(|X_{n,k}| \ge r_\delta - |z|), 
	\end{align*}
	where $r_\delta>0$ is a constant in the proof of Lemma \ref{l:add1}. Note that by induction hypothesis and definition of $r_\delta$ we have
	$$ \lim_{n \to \infty}\P\left(|M_n^{k-1}(z)| < e^{-n\eps}\right) = 0 \quad \mbox{and} \quad \limsup_{n \to \infty} \P\left(|X_{n,k}| \ge r_\delta - |z|\right) \le \frac{\delta}{2}.$$
	Set 
	$$ Y_n:= \frac{1}{z-X_{n,k}} \1_{\{|X_{n,k}| \le r_\delta -z\} }. $$
	Then by \eqref{c:add2}, we obtain
	$$ \theta_z(w) \le C_3 r_{\delta}^2 n^a $$
	for density $\theta_z(w)$ of $Y_n$, and
	$$ \P \left( \left| \frac{1}{z-X_{n,k}}+\frac{M^k_n(z)}{M^{k-1}_n(z)} \right| < e^{-n \eps} ; |X_{n,k}| \le r_\delta - |z| \right)  \le \sup_{x \in \C} \P(|Y_n+x| < e^{-n\eps}), $$
	which goes to $0$ as $n \rightarrow \infty$. Therefore, we have
	$$ \limsup_{n \to \infty}\P\left(|L^k_n(z)| < e^{-2n \eps}\right) \le \delta. $$
	Now lemma follows from the fact that $\delta$ is an arbitrary constant.
\end{proof}

\section{Application to 2D Coulomb gas ensembles}

In this section, we prove Theorem~\ref{Coulomb critical}. Recall that the joint probability density of 2D Coulomb gas ensemble is given as
\begin{equation*}
d\mathbf{P}_n^{\beta}(\zeta_1, \cdots, \zeta_n)=  
\frac{1}{ Z_n^{\beta}}  \prod_{j,k:j<k}|\zeta_j-\zeta_k|^{2\beta}e^{-\beta n \sum_j Q(\zeta_j)} d\mbox{vol}_{2n},
\end{equation*}
where $\beta>0$ is inverse temperature and $Q: \C \rightarrow \R$ is external potential satisfying the assumptions \textbf{(A1)} in Section 1.

First we recall some definitions and properties of equilibrium measure which we will use in this section.
For any probability measure $\mu$, logarithmic potential $U^{\mu} : \mathbb{C} \rightarrow (-\infty, \infty]$  is defined by 
$ U^{\mu}(\zeta):=\int_{\mathbb{C}}\log \frac{1}{|\zeta-\eta|^2}d\mu(\eta) $  
and logarithmic energy $I[\mu]$ is given as
$ I[\mu]:=\int_{\mathbb{C}^2}\log \frac{1}{|\zeta-\eta|^2}d\mu(\zeta)d\mu(\eta)= 
\int_{\mathbb{C}} U^{\mu}(\zeta) d\mu(\zeta).$ 
A subset $\mathcal{N}$ of $\C$ is said to be polar if $I[\mu]=\infty$ for all compactly supported probability measures with $\mbox{supp} \,(\mu) \in \mathcal{N}$. 
We say that some property holds {\it{quasi-everywhere (q.e)}} on  $E \subset \mathbb{C}$ if it holds everywhere on $E$ except some Borel polar set. Note that every Borel probability measure with finite logarithmic energy assigns zero Lebesgue measure
to Borel polar sets.  

For given admissible potential $Q$, the {\it{weighted logarithmic energy}} $I_Q[\mu]$ for each probability measure $\mu$ is defined as 
$$ 
I_Q[\mu]:=\iint_{\mathbb{C}^2}\log \frac{1}{|\zeta-\eta|^2}d\mu(\zeta)d\mu(\eta)+2\int_{\mathbb{C}} Q d\mu.
$$
The following theorem is a collection of properties of equilibrium measure.
\begin{thm} \label{eq msr} {\rm\cite[Chap I Theorem 1.3]{ST97}}
	
	Suppose that the potential $Q$ is admissible and let $$I_{Q} := \inf \left\{ I_Q[\mu] \right\}, $$ where infimum is over all probability measures on $\C$. 
	Then the following properties hold. 
	\begin{enumerate}
		\item  There exists the unique probability measure $ \sigma_{Q} $ such that $$I_Q \left[\sigma_{Q} \right]=I_{Q}. $$
		\item 
		$ I_Q\left[\sigma_{Q} \right]=I_Q$ is finite. 
		\item
		$I[\sigma_{Q}]=I_Q-2 \int_{\C}Qd\sigma_{Q}$  is finite. 
		\item
		$S_{Q}:= \mbox{supp} \,(\sigma_{Q}) $ is compact.
		\item
		Let 
		$$
		F_{Q}:=I_Q-\int Q d\sigma_{Q}.
		$$ 
		Then 
		$$ 
		U^{ \sigma_{Q}}(\zeta) +Q(\zeta) \ge F_{Q} \quad \text{holds for q.e.  } \zeta \in \C. 
		$$
		and 
		$$
		U^{\sigma_{Q}}(\zeta)+Q(\zeta) = F_{Q} \quad \text{holds for q.e.  } \zeta \in S_{Q}.
		$$ 
	\end{enumerate}
\end{thm}
The measure $ \sigma_{Q} $ is called the \textit{equilibrium measure} associated with $Q$. The constant $F_Q$ is called the \textit{modified Robin constant} for $Q$.	

Using the notion of equilibrium measure, we immediately obtain the following lemma. 
\begin{lem} For Lebesgue a.e. $\zeta \in \C$, there exists a positive constant $C>0$ satisfying
	\begin{equation}\label{ub condi den}
	\frac{ e^{-\beta n \left( Q(\zeta)-2\int_{\C}\log|\zeta-z| d\sigma_Q(z)  \right) }    }{ \int_{\C} e^{-\beta n \left( Q(\zeta)-2\int_{\C}\log|\zeta-z| d\sigma_Q(z)  \right) } d\zeta     } < C.
	\end{equation}
\end{lem}
\begin{proof}
	Note that by Theorem~\ref{eq msr}(3), the Borel polar set has Lebesgue measure zero. Therefore by  Theorem~\ref{eq msr}(5),  we have 
	$$
	Q(\zeta)-2\int_{\C}\log|\zeta-z| d\sigma_Q(z)=	U^{\sigma_{Q}}(\zeta)+Q(\zeta) \ge F_{Q},
	$$   
	for a.e. $\zeta \in \C$ and
	$$
	Q(\zeta)-2\int_{\C}\log|\zeta-z| d\sigma_Q(z)=	U^{\sigma_{Q}}(\zeta)+Q(\zeta) = F_{Q},
	$$  
	for a.e. $\zeta \in S_Q$, which implies 
	\begin{align*}
	\frac{ e^{-\beta n \left( Q(\zeta)-2\int_{\C}\log|\zeta-z| d\sigma_Q(z)  \right) }    }{ \int_{\C} e^{-\beta n \left( Q(\zeta)-2\int_{\C}\log|\zeta-z| d\sigma_Q(z)  \right) } d\zeta     } &\le \frac{ e^{-\beta n F_Q }    }{ \int_{S_Q} e^{-\beta n F_Q } d\zeta      } \le \frac{1}{m(S_Q)},
	\end{align*}
	where $m$ denotes the Lebesgue measure in $\C$. Therefore, Theorem~\ref{eq msr}(4) concludes the lemma.
\end{proof}

We denote by $W_p$ the \textit{Wasserstein distance} of order $p$.  
In particular, if $p=1$, by Kantorovich-Rubinstein dual representation, we have
$$
W_1(\mu, \nu)= \sup_{||f||_{\rm Lip} \le 1} \int f(x) (\mu-\nu) (dx), \quad ||f||_{\rm Lip} := \sup_{x \neq y} \frac{|f(x)-f(y)|}{|x-y|}.
$$ 
Let $\mu_n$ be the empirical measure of Coulomb gas ensemble, i.e., 
$\mu_n = \frac{1}{n} \sum_{j=1}^{n} \delta_{\zeta_{j}}.$ The following concentration inequality is due to Chafa\"i, Hardy and Ma\"ida.  
\begin{prop} \label{Concent ineq.}{\rm \cite[Theorem 1.5.]{CHM17}}
	There exists a constant $a'>0$ such that for any $n \ge 2$, and $r >0$,
	\begin{equation}\label{e:wd}
	\P_n^{\beta} \left( W_1(\mu_n, \sigma_Q)\ge r \right) \le e^{-a' n^2 r^2  }.
	\end{equation}
\end{prop}

Using this concentration inequality, we prove following lemma. 
\begin{lem} \label{condi den}
	There exists $\ve>0$ such that for Lebesgue a.e. $\zeta \in \C$,
	\begin{equation}\label{e:1}
	\lim_{n \to \infty} \P_n^{\beta} \left( \frac{e^{-\beta n Q(\zeta)} \prod_{j=1}^{n-1}  \left| \zeta-\zeta_j \right|^{2\beta} }{ \int_{\C} e^{-\beta n Q(\zeta)} \prod_{j=1}^{n-1}  \left| \zeta-\zeta_j \right|^{2\beta}d\zeta  } \le \exp \left( n^{1-\frac{\ve}{2}} \right) \right)= 1.
	\end{equation}
\end{lem}
\begin{proof}
	For some small $\eps>0$, set 
	\begin{align*}
	f_{n,\zeta}(z)&:=  \log (n^{1/2-2\eps}|z-\zeta|)_+ - (1/2 -2\eps) \log n; \\
	&= \begin{cases} \log |z-\zeta|, &\text{if}\quad |z-\zeta| \ge n^{-1/2+2\eps} \\ (-1/2+2\eps)\log n, &\text{if}\quad |z-\zeta| < n^{-1/2+2\eps}.  \end{cases}
	\end{align*}
	Using the concentration inequality \eqref{e:wd} with the choice $r = \frac{1}{2}n^{-1/2+\eps}$, we obtain
	$$\P_n^{\beta}\left(W_1(\mu_n, \sigma_Q) \ge \frac{1}{2}n^{-1/2+\eps} \right) \le e^{-2a n^{1+2\eps} }, $$
	for some positive constant $a>0$. Set $\mu'_n:= \frac{1}{n-1} \sum_{j=1}^{n-1} \delta_{\zeta_j}$. Then we have
	$$\P_n^{\beta}\left(W_1(\mu'_n, \sigma_Q) \ge n^{-1/2+\eps} \right) \le e^{-a n^{1+2\eps} }, $$
	for large $n$. Indeed, we have $W_1(\mu_n,\mu'_n) \le \frac{1}{n(n-1)} \sum_{j=1}^{n-1} |\zeta_j - \zeta_{n}| \to 0$ as $n \to \infty$ in probability, which follows from the tightness property of $\zeta_j$, see \cite[Theorem 1.12]{CHM17}. Thus, using $\| f_{n,\zeta} \|_{\rm Lip} = n^{1/2-2\eps}$ for every $\zeta \in \C$, we obtain
	$$ \P_n^{\beta} \left( \sup_{\zeta \in \C} \left| \int f_{n,\zeta}(z) \mu'_n(dz) - \int f_{n,\zeta}(z) \sigma_Q(dz) \right|  \ge n^{-\eps}     \right) \le e^{-a n^{1+2\eps} }. $$
	Also recall that conditions \textbf{(A1)}-(i),(iii) deduce the boundedness of the density of $\sigma_Q$. Thus, using the definition of $f_{n,\zeta}$ we have 
	\begin{align*}
	\int \left( f_{n,\zeta}(z)- \log|z-\zeta| \right) \sigma_Q(dz) \le \int_{B(\zeta,n^{-1/2+2\eps})}  -\log|z-\zeta| \sigma_Q(dz) \le c \, n^{-1+5\eps}.
	\end{align*}
	Combining these estimates with the fact that $f_{n,\zeta}(z) \ge \log|z-\zeta|$, we conclude that for large $n$,
	\begin{align*}
	&\P_n^{\beta} \left( \sup_{\zeta \in \C}  \int \log|z-\zeta| \mu'_n(dz) - \int \log|z-\zeta| \sigma_Q(dz) \ge 2n^{-\eps}     \right)  	\label{e:5.3.1}
	\\
	\le& \P_n^{\beta} \left( \sup_{\zeta \in \C}  \int \log|z-\zeta| \mu'_n(dz) - \int f_{n,\zeta}(z) \sigma_Q(dz)   \ge  2n^{-\eps} - c \, n^{-1+5\eps}  \right) , \nonumber
	\\
	\le& \P_n^{\beta}\left(\sup_{\zeta \in \C}   \int f_{n,\zeta}(z) \mu'_n(dz) - \int f_{n,\zeta}(z) \sigma_Q(dz) \ge n^{-\eps} \right) \le e^{-an^{1+2\eps}}. \nonumber
	\end{align*}
	Then by Theorem~\ref{eq msr}(5) and the fact that 
	\begin{align*}
	e^{-\beta nQ(\zeta)} \prod_{j=1}^{n} |\zeta - \zeta_j|^{2\beta} = \exp \left[ -\beta n \left\{Q(\zeta) + \int \log \frac{1}{|z-\zeta|^2} \mu_n(dz) \right\} \right],
	\end{align*}
	we have
	\begin{equation}\label{e:5.3.2}
	\P_n^{\beta} \left( \sup_{\zeta \in \C} e^{-\beta nQ(\zeta)} \prod_{j=1}^{n} |\zeta - \zeta_j|^{2\beta} \ge \exp(-n \beta F_Q + \beta n^{1-\eps}) \right) \le e^{-an^{1+2\eps}}.
	\end{equation}
	On the other hand, since  $\int_{\C} f_{n,\zeta}(z) - \log|z-\zeta|  \sigma_Q(d\zeta) \le  c \,n^{-1+5\eps} $, we have
	\begin{align} 
	\begin{split}\label{e:5.3}
	&\int_{S_Q} \int_{\C} f_{n,\zeta}(z) - \log|z-\zeta| \mu'_n(dz) \sigma_Q(d\zeta) \\
	&= \int_{\C} \int_{S_Q} f_{n,\zeta}(z) - \log|z-\zeta|   \sigma_Q(d\zeta) \mu'_n(dz),  \\
	&\le \int_\C cn^{-1+5\eps} \mu'_n(dz) = c \, n^{-1+5\eps}.
	\end{split}
	\end{align}
	Set 
	$$
	A:=\left\{ \zeta\in S_Q : \int_{\C} \left| \log|z-\zeta| - f_{n,\zeta}(z)\right| \mu'_n(dz) \ge n^{-1+6\eps} \right\}.
	$$
	Then by Chebyshev inequality and \eqref{e:5.3}, we have
	\begin{align*}
	\lim_{n \to \infty} \sigma_Q(A) \le \lim_{n \to \infty} \frac{c \, n^{-1+5\eps}}{n^{-1+6\eps}}=0.
	\end{align*}
  Note that	we may assume that $m(A) \le m(S_Q)/2$ since $\sigma_Q$ is absolutely continuous with respect to Lebesgue measure.  
	Therefore, we obtain that for large $n$,
	\begin{align*}
	&\P_n^{\beta} \left( \int_{\C} e^{-\beta n Q(\zeta)} \prod_{j=1}^{n-1}  \left| \zeta-\zeta_j \right|^{2\beta}d\zeta \le \int_{S_Q\sm A} e^{-\beta n F_Q - \beta n ^{1-\eps} - c\beta n^{5\eps} } d\zeta \right) \\ 
	\le& \P_n^{\beta} \left( \int_{S_Q \sm A} e^{-\beta n Q(\zeta)} \prod_{j=1}^{n-1}  \left| \zeta-\zeta_j \right|^{2\beta}d\zeta \le \int_{S_Q \sm A} e^{-\beta n F_Q - \beta n ^{1-\eps}- c\beta n^{5\eps} } d\zeta \right),
	\\
	\le& \P_n^{\beta} \left(  \inf_{\zeta \in \C} e^{-\beta n Q(\zeta)} \prod_{j=1}^{n-1}  \left| \zeta-\zeta_j \right|^{2\beta} \le e^{-\beta n F_Q - \beta n^{1-\eps} -c\beta n^{5\eps} } \right),
	\\ \le& \P_n^{\beta}\left(\sup_{\zeta \in \C} \left|  \int f_{n,\zeta}(z) \mu'_n(dz) - \int \log|z-\zeta|\sigma_Q(dz)\right| \ge n^{-\eps} + cn^{-1+5\eps} \right), \\
	\le& \P_n^{\beta}\left(\sup_{\zeta \in \C} \left|  \int f_{n,\zeta}(z) \mu'_n(dz) - \int f_{n,\zeta}(z) \sigma_Q(dz) \right| \ge n^{-\eps}  \right)
	\le e^{-an^{1+2\eps}}.
	\end{align*}   
	This and \eqref{e:5.3.2} proves the lemma.         
	
\end{proof}

\begin{proof}[Proof of Theorem~\ref{Coulomb critical}]
	Fix $k \in \N$. For $n \in \N$, let us denote by $\zeta_{n}, \dots, \zeta_{n}$ the $n$-th Coulomb gas ensembles. Set $z_{n,i}= \zeta_{i}$ for $1 \le i \le n-k$ and $Y_{n,i}=\zeta_{n-k+i}$ for $1 \le i \le k$. Recall that all we need to show is 
\begin{equation} \label{Coulomb eq 123}
 \lim_{n \to \infty} \P(|L_n^k(z)| < e^{-n\eps})= 0, \quad  L^k_n(z):=\frac{1}{k!} \frac{P^{(k)}_n(z)}{P_n(z)},
\end{equation}
where $P_n(z):=(z-\zeta_1)\cdots(z-\zeta_n)$.

For each $0 \le i \le k-1$, let $\mathcal{N}_{n,i}$ be the subset of sample space, which satisfies 
$$ \sup_{\zeta \in \C} \left[ \frac{e^{-\beta n Q(\zeta)} \prod_{j=1, j \neq n-i}^{n}  \left| \zeta-\zeta_j \right|^{2\beta} }{ \int_{\C} e^{-\beta n Q(z)} \prod_{j=1}^{n-1}  \left| z-\zeta_j \right|^{2\beta}dz  } \right] > \exp \left( n^{1-\frac{\epsilon}{2}} \right)
$$
and set $\mathcal{N}_n:= \cup_{i=0}^{k-1} \mathcal{N}_{n,i}$. Then by Lemma~\ref{condi den}, 
$$\P_n^{\beta}(\mathcal{N}_n) \le \sum_{i=0}^{k-1} \P_n^{\beta}(\mathcal{N}_{n,i}) \le k e^{-an^{1+2\eps}},
$$ 
which implies $\lim_{n \to \infty} \P_n^{\beta}(\mathcal N_n)=0$.

In the case of $\mathcal{N}^c_n$, we verify \eqref{Coulomb eq 123} as a consequence of Lemma~\ref{l:add2}. Therefore all we need to check is \eqref{c:add1}, \eqref{c:add2}, and \eqref{c:add3} for Coulomb gas ensemble. First note that \eqref{c:add2} is obtained from the construction of $\mathcal{N}_n$. To establish \eqref{c:add1}, notice that
	$$ \log_+ |\zeta| e^{-\beta n Q(\zeta)} \le e^{-\beta n Q_0(\zeta)}, $$
	where $Q_0(\zeta) = Q(\zeta) - |\zeta|$. Then by the assumption \textbf{(A1)}-(ii), the partition function of the Coulomb gas with potential $Q_0$ is also finite, which implies \eqref{c:add2}.
 Finally, \eqref{c:add3} follows from the well-known tightness of Coulomb  gas ensemble, see e.g., \cite[Theorem 1.12]{CHM17}.
Combining these, we obtain that 
	$$\lim_{n \to \infty} \P(|L_n^k(z)| < e^{-n\eps} ; \mathcal{N}_n^c) = 0. $$
	Therefore we conclude
	$$ \lim_{n \to \infty} \P(|L_n^k(z)| < e^{-n\eps}) \le \lim_{n \to \infty} [\P(|L_n^k(z)| < e^{-n\eps} ; \mathcal{N}_n^c) + \P(\mathcal N_n)] = 0,$$
	which completes the proof.
\end{proof}

\section{Questions}

\begin{enumerate}
	\item It is expected that the Theorem \ref{remove zero} hold for any probability measure $\mu$. Does the proof of Theorem \ref{remove zero} extend to the case when the limiting measure $\mu$ has atoms as well? 
	\item In the spirit of \cite{hanin1,hanin2}, one may ask if it is possible to show that most of the zeros have a critical point with in a distance of $O(\frac{1}{n})$. Establishing this would imply a natural pairing between zeros and critical points. Once the pairing is established, it will be of interest to study how the sum of pairwise distances (matching distance) behaves with $n$. This was studied in the case when all the zeros are real in \cite[Chapter-4]{Re16a}.
	\item Extending the previous question, it is pertinent to ask about the density of the distances between zeros and critical points in the scale of $\frac{1}{n}$. A particular case being the study of critical points of characteristic polynomial of Haar distributed unitary matrix. In this case one can notice that the law of zeros and hence for the critical points is rotationally invariant. The characteristic polynomial of CUE random matrix is believed to model Riemann zeta function and this problem is of interest in the study of critical points of Riemann zeta function. For more on this problem see \cite{CUE_Riemann_zeta} and references there in.  \\
\end{enumerate}

\noindent\textbf{Acknowledgments:} The authors would like to thank the organizers of the Second ZiF Summer School on Randomness in Physics and Mathematics, held at Bielefeld in August 2016, where this work was initiated.


\bibliographystyle{abbrv}
\bibliography{draft_bibliography}

\end{document}